\documentclass[11pt,draft]{article}
\usepackage{amsmath,amscd}
\usepackage{amssymb,latexsym,amsthm}
\usepackage{color}
\usepackage{amssymb}
\usepackage{color}
\usepackage{amsmath,amsthm,amscd}
\usepackage[latin1]{inputenc}
\usepackage{lscape}
\usepackage{fancyhdr}
\usepackage{amsfonts}
\usepackage{pb-diagram}

\numberwithin{equation}{section}
\newtheorem{theorem}{Theorem}[section]

\newtheorem{proposition}[theorem]{Proposition}

\newtheorem{corollary}[theorem]{Corollary}

\theoremstyle{definition}
\newtheorem{example}[theorem]{Example}
\newtheorem{definition}[theorem]{Definition}
\newtheorem{examples}[theorem]{Examples}

\newtheorem{remark}[theorem]{Remark}

\title{\textbf{Enveloping Algebra and Skew Calabi-Yau algebras\\ over Skew Poincar\'e-Birkhoff-Witt Extensions}}

\author{Armando Reyes\footnote{Departamento de Matem\'aticas. e-mail: mareyesv@unal.edu.co}\\ Universidad Nacional de
Colombia, Bogot\'a \\ H\'ector Su\'arez\footnote{Escuela de Matem\'aticas y Estad\'istica. e-mail: hector.suarez@uptc.edu.co}\\ Universidad Pedag\'ogica y Tecnol\'ogica de Colombia, Tunja}
\date{}
\begin{document}
\maketitle

\begin{abstract}
\noindent In this paper we show that the tensor product of skew PBW extensions is a skew PBW extension. We also characterize the enveloping algebra of a skew PBW extension. Finally, we establish sufficient conditions to guarantee the property of being skew Calabi-Yau algebra over skew PBW extensions.

\bigskip


\bigskip

\bigskip

\end{abstract}

\section{Introduction}
Let $k$ be a commutative ring and $B$ an associative $k$-algebra. By definition, the enveloping algebra of $B$ is the tensor product $B^e =B\otimes_k B^{\rm op}$ of $B$ with its opposite algebra $B^{\rm op}$. Bimodules over $B$ are
essentially the same as modules over the enveloping algebra of $B$, so
in particular $B$ can be considered as an $B^e$-module. Since that the opposite algebra of a Koszul algebra is also a Koszul algebra, and the tensor product of Koszul algebras is also a Koszul algebra, and having in mind that the authors have studied the property of being Koszul for skew Poincar\'e-Birkhoff-Witt (PBW for short) extensions (see  \cite{SuarezLezamaReyes2015},  \cite{SuarezReyes2017}, \cite{Suarez},  and \cite{SuarezReyes2017c}), in this paper we are interested in the characterization of the enveloping algebra of skew PBW extensions. These non-commutative rings of polynomial type were introduced in \cite{LezamaGallego}, and they are defined by a ring and a set of variables with relations
between them. Skew PBW extensions include rings and algebras coming from mathematical physics such PBW
extensions, group rings of polycyclic-by-finite groups, Ore algebras, operator algebras, diffusion
algebras, some quantum algebras, quadratic algebras in three variables, some 3-dimensional skew polynomial algebras, some quantum
groups, some types of Auslander-Gorenstein rings, some Calabi-Yau algebras, some quantum universal enveloping algebras, and others. A detailed list of examples can be consulted in \cite{ReyesPhD} and \cite{LezamaReyes2014}. Several ring, module and homological properties of these extensions have been studied (see for example \cite{ReyesPhD}, \cite{Reyes2013}, \cite{LezamaReyes2014}, \cite{Reyes2014}, \cite{Reyes2014UIS}, \cite{LezamaAcostaReyes2015},  \cite{Reyes2015}, \cite{ReyesSuarez2016a}, \cite{ReyesSuarez2016b}, \cite{ReyesSuarez2016c}, \cite{ReyesSuarez2017}, \cite{SuarezReyes2017}, \cite{Suarez},  \cite{SuarezReyes2017c}, and others). Besides our interest, it is important to say that the concept of enveloping algebra is of great importance in the research of another concepts in physics and mathematics (for instance, Calabi-Yau algebras  \cite{Wu2011}, \cite{Reyes}, \cite{Zhu2016}, see also Section \ref{application}).\\

The paper is organized as follows: In Section \ref{definitionexamplesspbw} we establish the necessary results about skew PBW extensions for the rest of the paper. Next, in Section \ref{tensorproduct} we establish some results about tensor product of skew PBW extensions. Finally, in Section \ref{envelopingalgebra}. Throughout the paper, the word ring means a ring not necessarily commutative with unity. The symbols $k$ and $\Bbbk$ will denote a commutative ring and a field, respectively.
\section{Skew PBW extensions}\label{definitionexamplesspbw}
In this section we recall the definition of skew PBW extension and present some of their properties. The proofs of these properties can be found in \cite{LezamaReyes2014}.
\begin{definition}[\cite{LezamaGallego}, Definition 1]\label{def.pbwextension}
Let $R$ and $A$ be rings. We say that $A$ \textit{is a skew PBW extension over} $R$ (also called a \textit{$\sigma$-PBW  extension of} $R$), if the following conditions hold:
\begin{enumerate}
\item[\rm (i)]$R\subseteq A$;
\item[\rm (ii)]there exist elements $x_1,\dots ,x_n\in A$ such that $A$ is a left free $R$-module, with basis the basic elements ${\rm Mon}(A):= \{x^{\alpha}=x_1^{\alpha_1}\cdots
x_n^{\alpha_n}\mid \alpha=(\alpha_1,\dots ,\alpha_n)\in
\mathbb{N}^n\}$.
\item[\rm (iii)]For each $1\leq i\leq n$ and any $r\in R\ \backslash\ \{0\}$, there exists an element $c_{i,r}\in R\ \backslash\ \{0\}$ such that $x_ir-c_{i,r}x_i\in R$.
\item[\rm (iv)]For any elements $1\leq i,j\leq n$,  there exists $c_{i,j}\in R\ \backslash\ \{0\}$ such that $x_jx_i-c_{i,j}x_ix_j\in R+Rx_1+\cdots +Rx_n$.
\end{enumerate}
Under these conditions, we will write $A:=\sigma(R)\langle
x_1,\dots,x_n\rangle$.
\end{definition}
\begin{remark}[\cite{LezamaGallego}, Remark 2]\label{notesondefsigampbw}
\begin{enumerate}
\item [\rm (i)] Since ${\rm Mon}(A)$ is a left $R$-basis of $A$, the elements $c_{i,r}$ and $c_{i,j}$ in Definition \ref{def.pbwextension} are unique.
\item [\rm (ii)] In Definition \ref{def.pbwextension} (iv), $c_{i,i}=1$. This follows from $x_i^2-c_{i,i}x_i^2=s_0+s_1x_1+\cdots+s_nx_n$, with $s_i\in R$, which implies $1-c_{i,i}=0=s_i$.
\item [\rm (iii)] Let $i<j$. By Definition \ref{def.pbwextension} (iv),   there exist elements $c_{j,i},c_{i,j}\in R$ such that $x_ix_j-c_{j,i}x_jx_i\in R+Rx_1+\cdots +Rx_n$ and $x_jx_i-c_{i,j}x_ix_j\in R+Rx_1+\cdots+Rx_n$, and hence $1=c_{j,i}c_{i,j}$, that is, for each $1\leq i<j\leq n$, $c_{i,j}$ has a left inverse and $c_{j,i}$ has a right inverse. In general, the elements $c_{i,j}$ are not two sided invertible. For instance, $x_1x_2=c_{2,1}x_2x_1+p=c_{21}(c_{1,2}x_1x_2+q)+p$, where $p,q\in R+Rx_1+\dotsb + Rx_n$, so $1=c_{2,1}c_{1,2}$, since $x_1x_2$ is a basic element of ${\rm Mon}(A)$. Now, $x_2x_1=c_{1,2}x_1x_2+q=c_{1,2}(c_{2,1}x_2x_1+p)+q$, but we cannot conclude that $c_{12}c_{21}=1$ because $x_2x_1$ is not a basic element of ${\rm Mon}(A)$ (we recall that ${\rm Mon}(A)$ consists of the standard monomials).\
\item [\rm (iv)] Every element $f\in A\ \backslash\ \{0\}$ has a unique representation in the form $f=c_0 + c_1X_1+\cdots+c_tX_t$, with $c_i\in R\ \backslash\ \{0\}$ and $X_i\in {\rm Mon}(A)$, for $1\leq i\leq t$.
\end{enumerate}
\end{remark}
\begin{proposition}[\cite{LezamaGallego}, Proposition
3]\label{sigmadefinition}
Let $A$ be a skew PBW  extension over $R$. For each $1\leq i\leq
n$, there exist an injective endomorphism $\sigma_i:R\rightarrow
R$ and an $\sigma_i$-derivation $\delta_i:R\rightarrow R$ such that $x_ir=\sigma_i(r)x_i+\delta_i(r)$, for every $r\in R$.
\end{proposition}
\begin{definition}[\cite{LezamaGallego}, Definition 6]\label{definitioncoefficients}
Let $A$ be a skew PBW  extension over $R$ with endomorphisms
$\sigma_i$, $1\leq i\leq n$, as in Proposition
\ref{sigmadefinition}.
\begin{enumerate}
\item[\rm (i)]For $\alpha=(\alpha_1,\dots,\alpha_n)\in \mathbb{N}^n$,
$\sigma^{\alpha}:=\sigma_1^{\alpha_1}\cdots \sigma_n^{\alpha_n}$,
$|\alpha|:=\alpha_1+\cdots+\alpha_n$. If
$\beta=(\beta_1,\dots,\beta_n)\in \mathbb{N}^n$, then
$\alpha+\beta:=(\alpha_1+\beta_1,\dots,\alpha_n+\beta_n)$.
\item[\rm (ii)]For $X=x^{\alpha}\in {\rm Mon}(A)$,
$\exp(X):=\alpha$ and $\deg(X):=|\alpha|$. The symbol $\succeq$ will denote a total order defined on ${\rm Mon}(A)$ (a total order on $\mathbb{N}_0^n$). For an
 element $x^{\alpha}\in {\rm Mon}(A)$, ${\rm exp}(x^{\alpha}):=\alpha\in \mathbb{N}_0^n$.  If
$x^{\alpha}\succeq x^{\beta}$ but $x^{\alpha}\neq x^{\beta}$, we
write $x^{\alpha}\succ x^{\beta}$. If $f=c_1X_1+\dotsb +c_tX_t\in
A$, $c_i\in R\ \backslash\ \{0\}$, with $X_1\prec \dotsb \prec X_t$, then ${\rm
lm}(f):=X_t$ is the \textit{leading monomial} of $f$, ${{\rm
lc}}(f):=c_t$ is the \textit{leading coefficient} of $f$, ${\rm
lt}(f):=c_tX_t$ is the \textit{leading term} of $f$,  ${\rm exp}(f):={\rm exp}(X_t)$ is the \textit{order} of $f$, and $E(f):=\{{\rm exp}(X_i)\mid 1\le i\le t\}$. Finally, if $f=0$, then
${\rm lm}(0):=0$, ${\rm lc}(0):=0$, ${\rm lt}(0):=0$. We also
consider $X\succ 0$ for any $X\in {\rm Mon}(A)$. For a detailed description of monomial orders in skew PBW  extensions, see \cite{LezamaGallego}, Section 3.
\item[\rm (iii)]If $f$ is an element as in Remark  \ref{notesondefsigampbw} (iv), then $\deg(f):=\max\{\deg(X_i)\}_{i=1}^t$.
\end{enumerate}
\end{definition}
Skew PBW extensions are  characterized in the following way.
\begin{proposition}[\cite{LezamaGallego}, Theorem 7]\label{coefficientes}
Let $A$ be a polynomial ring over $R$ with respect to a set of indeterminates  
$\{x_1,\dots,x_n\}$. $A$ is a skew PBW extension over $R$ if and only if the following conditions are satisfied:
\begin{enumerate}
\item[\rm (i)]for each $x^{\alpha}\in {\rm Mon}(A)$ and every $0\neq r\in R$, there exist unique elements $r_{\alpha}:=\sigma^{\alpha}(r)\in R\ \backslash\ \{0\}$, $p_{\alpha ,r}\in A$ such that $x^{\alpha}r=r_{\alpha}x^{\alpha}+p_{\alpha , r}$,  where $p_{\alpha ,r}=0$, or $\deg(p_{\alpha ,r})<|\alpha|$ if
$p_{\alpha , r}\neq 0$. If $r$ is left invertible,  so is $r_\alpha$.
\item[\rm (ii)]For each $x^{\alpha},x^{\beta}\in {\rm Mon}(A)$, there exist unique elements $c_{\alpha,\beta}\in R$ and $p_{\alpha,\beta}\in A$ such that
$x^{\alpha}x^{\beta} = c_{\alpha,\beta}x^{\alpha+\beta}+p_{\alpha,\beta}$, where $c_{\alpha,\beta}$ is left invertible, $p_{\alpha,\beta}=0$, or $\deg(p_{\alpha,\beta})<|\alpha+\beta|$ if
$p_{\alpha,\beta}\neq 0$.
\end{enumerate}
\end{proposition}
There are some examples of skew PBW extensions which are very important for several results in the paper (see Propositions \ref{inocente}, \ref{opbijectiveprop}, and Theorem \ref{teo.skew imp skew}). This is the content of the following definition.
\begin{definition}\label{sigmapbwderivationtype}
Let $A$ be a skew PBW extension over $R$, $\Sigma:=\{\sigma_1,\dotsc, \sigma_n\}$ and $\Delta:=\{\delta_1,\dotsc, \delta_n\}$, where $\sigma_i$ and $\delta_i$ ($1\leq i\leq n$) are as in Proposition \ref{sigmadefinition}.
\begin{enumerate}
\item[\rm (a)]\label{def.constant} Any element $r$ of $R$ such that $\sigma_i(r)=r$ and $\delta_i(r)=0$, for all $1\leq i\leq n$, it is  called a {\em constant}. $A$ is called \emph{constant} if every element of $R$ is constant.
\item[\rm (b)]\label{def.quasicom}  $A$ is called \textit{quasi-commutative} if the conditions
{\rm(}iii{\rm)} and {\rm(}iv{\rm)} in Definition
\ref{def.pbwextension} are replaced by the following conditions: (iii') for each $1\leq i\leq n$ and every $r\in R\ \backslash\ \{0\}$, there exists $c_{i,r}\in R\ \backslash\ \{0\}$ such that $x_ir=c_{i,r}x_i$; (iv') for any $1\leq i,j\leq n$, there exists $c_{i,j}\in R\ \backslash\ \{0\}$ such that $x_jx_i=c_{i,j}x_ix_j$.
\item[\rm (c)] $A$ is called \textit{bijective}, if $\sigma_i$ is bijective for each $\sigma_i\in \Sigma$, and the elements $c_{i,j}$ are invertible for any $1\leq
i<j\leq n$. The importance of the restriction on the injective endomorphisms $\sigma_i$ is precisely due to the fact that under this condition we have $Rx_i + R = x_iR + R$, and hence every element $f=c_0 + c_1X_1+\cdots+c_tX_t\in A$ (Remark \ref{notesondefsigampbw} (iv)),  can be rewritten in the reverse form $f=c'_0 + X_1c'_1 + \dotsb + X_tc'_t$. In other words, if the functions $\sigma_i$ are bijective, then $A_R$ is a right free $R$-module with basis ${\rm Mon}(A)$  (\cite{LezamaReyes2014}, Proposition 1.7). In fact, a lot of properties (Noetherianess, regularity, Serre's theorem, global homological dimension, Gelfand-Kirillov dimension, Goldie dimension, semisimple Jacobson, prime ideals, Quillen's $K$-groups of higher algebraic $K$-theory, Baerness, quasi-Baerness, Armendariz, etc) of skew PBW extensions have been studied using this assumption of bijectivity (see \cite{ReyesPhD}, \cite{Reyes2013},  \cite{LezamaReyes2014}, \cite{Reyes2014}, \cite{Reyes2014UIS}, \cite{Reyes2015}, \cite{ReyesSuarez2016a}, \cite{SuarezReyes2017}, \cite{Suarez},     \cite{SuarezReyes2017c}, and others). 
\end{enumerate}
\end{definition}
\section{Tensor product of skew PBW extensions}\label{tensorproduct}
\begin{proposition}{\bf (Change of scalars)}\label{inocente}
If $A$ is a skew PBW extension over $k$ and $B$ is a commutative $k$-algebra, then $B\otimes_k A$ is a skew PBW extension over $B\otimes_k k$, that is, an extension over $B$.
\begin{proof}
We show that the four conditions established in Definition \ref{def.pbwextension} are satisfied. First of all, it is clear that $B\otimes_k k \subseteq B\otimes_k A$. Second of all, we know that $B\otimes_k A$ is an $B$-algebra under the product $b'(b\otimes a):=b'b\otimes a$, and $B\otimes_k A$ is left $B$-free with the same rank of $A$ as $k$-module (\cite{Rowen2008}, Remark 18.27), and since $B\cong B\otimes_k k$, it follows that $B\otimes_k A$ is a left $B\otimes_k k$-free module. Note that
\[
{\rm Mon}(B\otimes_k A) = \{(1\otimes x_1)^{\alpha_1} \dotsb (1\otimes x_n)^{\alpha_n}\mid \alpha_i\in \mathbb{N}, 1\le i\le n\},
\]
which shows that $B\otimes_k A$ is left $B\otimes_k k$-free of the same rank of $A$ as $k$-module. Now, if $r\otimes k'$ is a non-zero element of $B\otimes_k k$ (that is, the element $rk'$ since $B\cong B\otimes_k k$), we can see that for every element $1\otimes x_i$ of the basis, there exists a non-zero element $1\otimes c_{i,k'}\in B\otimes_k k$ such that
\[
(1\otimes x_i)(r\otimes k') - (1\otimes c_{i,k'})(1\otimes x_i) \in B\otimes_k k,
\]
where $c_{i,r}\in k$ satisfies $x_ir-c_{i,r}x_i \in k$, since $A$ is a skew PBW extension over $k$. Finally, the condition
\[
(1\otimes x_i)(1\otimes x_j) - (1\otimes c_{i,j})(1\otimes x_j)(1\otimes x_i) \in B\otimes_k k + (B\otimes_k k)(1\otimes x_1) + \dotsb + (B\otimes_k k)(1\otimes x_n),
\]
it follows from Definition \ref{def.pbwextension} (iv) applied to the extension $A$.
\end{proof}
\end{proposition}
\begin{remark}
The injective endomorphisms and the derivations for the skew PBW extension mentioned in the Proposition \ref{inocente} are given by $\sigma_i^{\otimes}:B\otimes_k k\to B\otimes_k k,\ \sigma_i^{\otimes}(b\otimes r):= b\otimes \sigma_i(r)$, and $\delta_i^{\otimes}:B\otimes_k k \to B\otimes_k k,\ \delta_i^{\otimes}(b\otimes r):= b\otimes \delta_i(r)$, respectively. It is straightforward to see that the functions $\sigma_i^{\otimes}$ are actually injective endomorphisms and that the functions $\delta^{\otimes}_i$ are $\sigma_i^{\otimes}$-derivations, for every $1\le i\le n$.
\end{remark}
\begin{examples}
If $A=\sigma(k)\langle x_1,\dotsc, x_n\rangle$, and $B=k[y]=k[y_1,\dotsc, y_m]$, then
\[
k[y]\otimes_k A\cong \sigma (k[y]\otimes_k k)\langle 1\otimes x_1,\dotsc, 1\otimes x_n\rangle\cong \sigma(k[y])\langle z_1,\dotsc, z_n\rangle.
\]
Let us see some examples of remarkable non-commutative rings which illustrated this isomorphism. A detailed reference of every example can be found in \cite{ReyesPhD} or   \cite{LezamaReyes2014}.
\begin{enumerate}
\item [\rm (a)] \label{additivean} \label{additivean} \textit{Additive analogue of the Weyl algebra.} This algebra is the $\Bbbk$-algebra $A_n(q_1,\dots,q_n)$ generated by the indeterminates $x_1,\dots,x_n,$ $y_1,\dots,y_n$ subject to the relations:
\begin{align*}
x_jx_i& = x_ix_j, & 1 \leq i,j \leq n,\\
y_jy_i&= y_iy_j, & 1 \leq i,j \leq n,\\
y_ix_j&=x_jy_i,& i\neq j,\\
y_ix_i&= q_ix_iy_i + 1,& 1\leq i\leq n,
\end{align*}
where $q_i\in \Bbbk\ \backslash\ \{0\}$. From \cite{LezamaReyes2014}, Example 3.5 (a), we have the isomorphisms $A_n(q_1, \dots, q_n)\cong \sigma(\Bbbk)\langle x_1,\dotsc, x_n; y_1,\dotsc, y_n\rangle \cong \sigma(\Bbbk[x_1,\dotsc,x_n])\langle
y_1,\dots,y_n\rangle$, that is, $A_n(q_1,\dots,q_n)$ is a skew PBW extension of the field $\Bbbk$ or the polynomial ring $\Bbbk[x_1,\dotsc,x_n]$. Now, by Proposition \ref{inocente} we obtain the isomorphism $\Bbbk[x_1,\dotsc,x_n]\otimes_{\Bbbk} \sigma(\Bbbk)\langle y_1,\dotsc, y_n\rangle\cong \sigma(\Bbbk[x_1,\dotsc,x_n])\langle 1\otimes y_1,\dotsc, y_n\rangle$, whence $A_n(q_1,\dotsc, q_n)\cong \Bbbk[x_1,\dotsc,x_n]\otimes_{\Bbbk} \sigma(\Bbbk)\langle y_1,\dotsc, y_n\rangle$.
\item [\rm (b)] \textit{Multiplicative analogue of the Weyl algebra}. By definition, this non-commutative ring is the $\Bbbk$-algebra $\mathcal{O}_n(\lambda_{ji})$ generated by
$x_1,\dots,x_n$ satisfying $x_jx_i
=\lambda_{ji}x_ix_j ,\ 1\leq i<j\leq n$, $\lambda_{ji}\in
\Bbbk\ \backslash\ \{0\}$. It can be proved that $\mathcal{O}_n(\lambda_{ji})\cong\sigma(\Bbbk)\langle
x_1,\dotsc,x_n\rangle\cong \sigma(\Bbbk[x_1])\langle
x_2,\dotsc,x_n\rangle$ (\cite{LezamaReyes2014}, Example 3.5 (a)).  Now, Proposition \ref{inocente} guarantees that $\Bbbk[x_1]\otimes_{\Bbbk} \sigma(\Bbbk)\langle x_2,\dotsc, x_n\rangle\cong \sigma(\Bbbk[x_1])\langle 1\otimes x_2,\dotsc,1\otimes x_n\rangle$, and so $\mathcal{O}_n(\lambda_{ji})\cong \Bbbk[x_1]\otimes_{\Bbbk} \sigma(\Bbbk)\langle x_2,\dotsc, x_n\rangle$.
\end{enumerate}
\end{examples}

The next proposition treats the construction of skew PBW extensions over the same ring of coefficients. Note that if $A$ is a skew PBW extension over a ring $R$, then $A$ is a right $R$-module under the multiplication in $A$, that is, $f\cdot r:=fr,\ f\in A, r\in R$. However, $A$ is not necessarily a right free $R$-module; in fact, if $A$ is bijective, then $A_R$ is free with basis the set ${\rm Mon}(A)$ established in Definition \ref{def.pbwextension} (ii) (see \cite{LezamaReyes2014}, Proposition 1.7 for a detailed proof of this fact).
\begin{proposition}\label{tensorproductsamering}
If $A=\sigma(R)\langle x_1,\dotsc, x_n \rangle$ and $A'=\sigma(R)\langle y_1, \dotsc, y_m\rangle$ are two  skew PBW extensions over $R$, then $A\otimes_R A'$ is also a skew PBW extension over $R$, and we have $A\otimes_R A'=\sigma(R)\langle x_1\otimes 1,\dotsc, x_n\otimes 1, 1\otimes y_1, \dots, 1\otimes y_m\rangle$.
\begin{proof}
Again, let us illustrate the four conditions of Definition \ref{def.pbwextension}. It is clear that $R\subseteq A\otimes_R A'$. Now, since the product of left free $R$-modules is a left free $R$-module with $R$-basis $\{x^{\alpha}\otimes y^{\beta}\mid \alpha\in \mathbb{N}^{n},\ \beta\in \mathbb{N}^{m}\}$, and having in mind that 
{\small{\begin{align*}
\{x^{\alpha}\otimes y^{\beta}\mid \alpha\in \mathbb{N}^{n},\ \beta\in \mathbb{N}^{m}\} = &\ \{x_1^{\alpha_1}\dotsb x_n^{\alpha_n}\otimes y_1^{\beta_1}\dotsb y_m^{\alpha_m}\mid \alpha_i,\ \beta_j\in \mathbb{N}\}\\
= &\ \{(x_1^{\alpha_1}\otimes 1) \dotsb (x_n^{\alpha_n}\otimes 1) (1\otimes y_1^{\beta_1})\dotsb  (1\otimes y_m^{\beta_m})\mid \alpha_i,\ \beta_j\in \mathbb{N}\}\\
= &\ \{(x_1\otimes 1)^{\alpha_1}\dotsb (x_n\otimes 1)^{\alpha_n} (1\otimes y_1)^{\beta_1}\dotsb  (1\otimes y_m)^{\beta_m}\mid \alpha_i,\ \beta_j\in \mathbb{N}\},
\end{align*}}}
then the set 
\[
{\rm Mon}(A\otimes A'):=\{(x_1\otimes 1)^{\alpha_1}\dotsb (x_n\otimes 1)^{\alpha_n}(1\otimes y_1)^{\beta_1}\dotsb  (1\otimes y_m)^{\beta_m}\mid \alpha_i,\ \beta_j\in \mathbb{N}\},
\] 
is the $R$-basis for the left free $R$-module $A\otimes A'$.

For a non-zero element $r$ of $R$, we have $(x_i\otimes 1)r - c_{i,r}(x_i\otimes 1)\in R\otimes 1$, $(1\otimes y_j)r - c_{j,r}(1\otimes y_j)\in 1\otimes R$, for $(1\le i\le n, 1\le j\le m)$, because $A$ and $A'$ are skew PBW extensions of $R$.

Now, note that 
\begin{align*}
(x_j\otimes 1)(x_i\otimes 1) - c_{i,j}(x_i\otimes 1)(x_j\otimes 1) \in R\otimes 1 + \sum_{l=1}^{n} R(x_l\otimes 1) + \sum_{p=1}^{m} R(1\otimes y_p), \ 1\le i, j \le n\\
(1\otimes y_j)(1\otimes y_i) - c'_{i,j}(1\otimes y_i)(1\otimes y_j) \in 1\otimes R + \sum_{l=1}^{n} R(x_l\otimes 1) + \sum_{p=1}^{m} R(1\otimes y_p), \ 1\le i, j \le m\\
(x_i\otimes 1),
\end{align*}
where the elements $c_{i,j}, c'_{i,j}\in R$ are considered from Definition \ref{def.pbwextension} (iv) for the extensions $A$ and $A'$, respectively. Finally,  we impose the relations $(1\otimes y_j) (x_i\otimes 1) - (x_i\otimes 1)(1\otimes y_j) = 0$, $(1\le i \le n,\ 1\le j\le m)$, with the aim of guarantee  the condition (iv) of the definition of the skew PBW extension $A\otimes_R A'$ over $R$.
\end{proof}
\end{proposition}
\begin{remark}\label{dogma}
If $A$ and $A'$ are two skew PBW extensions over $R$ as in Proposition \ref{tensorproductsamering}, then the injective endomorphisms $\sigma_i^{\otimes}$ of $R$, and the $\sigma_i^{\otimes}$-derivations $\delta_i^{\otimes}$ of $R$ for the extension $A\otimes_R A'$, are obtained using the injective endomorphisms  $\sigma_i$ and $\sigma_j'$, and the  $\sigma_i$-derivations, and $\sigma_j'$-derivations of the extensions $A$  and  $A'$, respectively. Therefore, the Proposition \ref{tensorproductsamering} can be established in the following way: the tensor product of two skew PBW extensions $A=\sigma(R)\langle x_1,\dotsc, x_n\rangle$ and $A'=\sigma(R)\langle y_1,\dotsc, y_m\rangle$, given by the relations
\begin{align*}
x_ir = c_{i,r}x_i + \delta_i(r),\quad \quad &\ 1\le i\le n,\\
x_jx_i-c_{i,j}x_ix_j\in R+\sum_{l=1}^{n}  Rx_l,\quad \quad &\ 1\le i, j \le n,
\end{align*}
and,
\begin{align*}
y_jr = c'_{j,r}y_j + \delta_j'(r),\quad \quad &\ 1\le i\le m,\\
y_jy_i-c'_{i,j}y_iy_j\in R+\sum_{p=1}^{m}  Ry_p,\quad \quad &\ 1\le i, j \le m,
\end{align*}
respectively, it is the left free algebra
\[
A\otimes A' = R\langle x_1,\dotsc x_n \rangle / I,
\]
where $R\langle x_1,\dotsc, x_n\rangle$ is the free $R$-algebra, and $I$ is the left ideal generated by the relations
\begin{align*}
(x_i\otimes 1)r - c_{i,r}(x_i\otimes 1) - \delta_i(r)\otimes 1,\quad\quad &\ 1\le i\le n \\  (1\otimes y_j)r - c_{j,r}(1\otimes y_j) - 1\otimes \delta_j'(r),\quad\quad &\ 1\le j\le m\\
(x_j\otimes 1)(x_i\otimes 1) - c_{i,j}(x_i\otimes 1)(x_j\otimes 1) + R\otimes 1 + \sum_{l=1}^{n} R(x_l\otimes 1),\quad\quad &\ 1\le i, j \le n\\
(1\otimes y_j)(1\otimes y_i) - c'_{i,j}(1\otimes y_i)(1\otimes y_j) + 1\otimes R + \sum_{p=1}^{m} R(1\otimes y_p),\quad\quad &\ 1\le i, j \le m\\
(1\otimes y_j) (x_i\otimes 1) - (x_i\otimes 1)(1\otimes y_j), \quad \quad &\ 1\le i \le n,\ 1\le j\le m.
\end{align*}
\end{remark}
From Proposition \ref{tensorproductsamering} it follows the next result.
\begin{corollary}
If $\{A_i\}_{i\in I}$ is a family of skew PBW extensions over the ring $R$, then $\bigotimes\limits_{i\in I} A_i$ is also a skew PBW extension of $R$.
\end{corollary}
\begin{example}
The Weyl algebra $A_{n+m}(k)$ is the skew PBW extension $A_n(k)\otimes_k A_m(k)$, which can be obtained using Proposition \ref{tensorproductsamering}. More precisely, since the Weyl algebra  $A_n(k)$ is the left free $k$-algebra $k\langle x_1, \dotsc, x_n, y_1,\dotsc, y_n\rangle$ with ideal of relations generated by $x_jx_i - x_ix_j,\ y_jx_i - x_iy_j-\delta_{ij},\ y_jy_i - y_iy_j$, for $1\le i < j \le n$, and similarly $A_m(\Bbbk)$, then
\begin{align*}
A_n(k)\otimes_k A_m(k) = &\ k\langle (x_1\otimes 1), \dotsc, (x_n\otimes 1), (y_1\otimes 1), \dotsc, (y_n\otimes 1),\\ &\ (1\otimes x_1'), \dotsc, (1\otimes x_m'), (1\otimes y_1'), \dotsc, (1\otimes y_m')\rangle / I,
\end{align*}
where $I$ is the left ideal generated  by the relations
\begin{align*}
I= &\ \langle (x_j\otimes 1)(x_i\otimes 1) - (x_i\otimes 1)(x_j\otimes 1), &\ 1\le i < j\le n\\
&\ (y_j\otimes 1)(x_i\otimes 1) -(x_i\otimes 1)(y_j\otimes 1) - \delta_{ij}\otimes 1, &\ 1\le i < j \le n\\
&\ (y_j\otimes 1)(y_i\otimes 1) - (y_i\otimes 1)(y_j\otimes 1), &\ 1\le i < j\le n,\\
&\ (1\otimes x_j')(x_i\otimes 1) - (x_i\otimes 1)(1\otimes x_j'),&\ 1\le i \le n,\ 1\le j \le m,\\
&\ (1\otimes y_j')(x_i\otimes 1) - (x_i\otimes 1)(1\otimes y_j'), &\ 1\le i\le n,\ 1\le j\le m,\\
&\ (1\otimes x_j')(y_i\otimes 1) - (y_i\otimes 1)(1\otimes x_j'), &\  1\le i\le n,\ 1\le j \le m,\\
&\ (1\otimes y_j')(y_i\otimes 1) - (y_i\otimes 1)(1\otimes y_j'), &\ 1\le i\le n,\ 1\le j\le m,\\
&\ (1\otimes x_j')(1\otimes x_i') - (1\otimes x_i')(1\otimes x_j'), &\ 1\le i < j\le m,\\
&\ (1\otimes y_j')(1\otimes x_i') - (1\otimes x_i')(1\otimes y_j') - 1\otimes \delta_{ij}, &\ 1\le i< j \le m,\\
&\ (1\otimes y_j')(1\otimes y_i') - (1\otimes y_i')(1\otimes y_j'), &\ 1\le i < j \le m.
\end{align*}
If we identify $p_i:=x_i\otimes 1$ $(1\le i\le n)$, $p_{n+i}:=1\otimes x_i'$ $(1\le i\le m)$, $q_i:= y_i\otimes 1$ $(1\le i\le n)$, $q_{n+i}:=1\otimes y_i'$ $(1\le i\le m)$, then we can see that the algebra $A_n(k)\otimes_k A_m(k)$ is precisely the Weyl algebra $A_{n+m}(k)$.
\end{example}
Next, we study the tensor product of skew PBW extensions whose coefficients rings are not necessarily the same. In this way, we generalize Proposition \ref{tensorproductsamering}.
\begin{proposition}
If $A=\sigma(R)\langle x_1,\dotsc, x_n\rangle$ and $A'=\sigma(R')\langle x_1,\dotsc, x_n\rangle$ are two skew PBW extensions over the $k$-algebras $R$ and $R'$ respectively, then  $A\otimes_k A'$ is a skew PBW extension over $R\otimes_k R'$.
\begin{proof}
Let $A$ and $A'$ be skew PBW extensions over the $k$-algebras $R$ and $R'$, respectively. From the definition we know that 
\begin{align*}
x_ir = &\ \sigma_i(r)x_i+\delta_i(r), \ \ \ 1\le i\le n, \ \ \ x_jx_i-c_{i,j}x_ix_j \in R + \sum_{l=1}^{n} Rx_l, \ \ 1\le i,  j\le n,\\
y_js = &\ \sigma'_j(s)y_j + \delta'_j(s),\ \ 1\le j \le m,\ \ \
y_jy_i - d_{i,j}y_iy_j \in R' + \sum_{l=1}^{m}
R'y_l,\ \ \ \ 1\le i, j\le m,
\end{align*}
where $\sigma_i, \delta_i: R\to R$ and $\sigma'_j, \delta'_j: R'\to R'$ are as in Proposition \ref{sigmadefinition}. We assume that the elements of $k$ commute with every element of $A$ and each element of $A'$, so $A$ and $A'$ are $k$-algebras (this assumption, for example, it was used in the computation of Gelfand-Kirillov dimension for these non-commutative rings, see \cite{Reyes2013}). Note that  $A\otimes_k A'$ is a $k$-algebra with the product given by $(a\otimes a')(b\otimes b')=(ab)\otimes (a'b')$ (\cite{Rotman2009}, Proposition 2.60).  Moreover, $R\cong R\otimes_k k$,\ $R'\cong k\otimes_k R',\ A\cong A\otimes_k k$, and $A'\cong k\otimes_k A'$. We endow $A\otimes_k A'$ with the  natural structure of left $R\otimes_k R'$-module, i.e., $(r\otimes s)\cdot (a\otimes a'):=(ra)\otimes (sa')$.\\

With the aim of showing that $A\otimes_k A'$ is a skew PBW extension of  $R\otimes_k R'$, we consider the free algebra $(R\otimes_k R')/ I$, where $I$ is the left ideal generated by the relations
\begin{align*}
&\ (x_i\otimes 1)(r\otimes 1) - (\sigma_i(r)\otimes 1)(x_i\otimes 1) - \delta_i(r)\otimes 1, &\ 1\le i\le n\\
&\ (x_j\otimes 1)(x_i\otimes 1) -c_{i,j}(x_i\otimes 1)(x_j\otimes 1) - R\otimes 1 + \sum_{l=1}^{n} R(x_l\otimes 1), &\  1\le i, j\le n\\
&\ (1\otimes y_j)(1\otimes s) - (1\otimes \sigma_j'(s))(1\otimes y_j) - 1\otimes \delta_i'(s), &\ 1\le j\le m\\
&\ (1\otimes y_j)(1\otimes y_i) -d_{i,j}(1\otimes y_i)(1\otimes y_j) - 1\otimes S + \sum_{l=1}^{m} S(1\otimes y_l), &\  1\le i, j\le m\\
&\ (x_i\otimes 1)(1\otimes s) - (1\otimes s)(x_i\otimes 1), &\ 1\le i\le n\\
&\ (1\otimes y_j)(r\otimes 1) - (r\otimes 1)(1\otimes y_j), &\ 1\le j\le s\\
&\ (x_i\otimes 1)(1\otimes y_j) - (1\otimes y_j)(x_i\otimes 1), &\ 1\le i\le n,\ 1\le j\le m.
\end{align*}
It is clear that $R\otimes_k R'\subseteq A\otimes_k A'$. Since $A\otimes_k A'$ is a left $R\otimes_k R'$-module, following the notation established in Definition \ref{definitioncoefficients}, and using Remark \ref{notesondefsigampbw} (iv), we can see that $A\otimes_k A'$ is left free over $R\otimes_k R'$ (the proof is similar to the  established in \cite{Roman2008}, Theorem 14.5, and uses some arguments about the union of sets of basic monomials (see  \cite{Reyes2014}, Lemma 4.3, for details about this procedure) with basis
\[
{\rm Mon}(A\otimes A'):=\{(x_1\otimes 1)^{\alpha_1} \dotsb (x_n\otimes 1)^{\alpha_n}(1\otimes y_1)^{\alpha_{n+1}} \dotsb (1\otimes y_m)^{\alpha_{n+m}}\}
\]
The injective endomorphims and the derivations of the skew PBW extension $A\otimes_k A$ are given by
\[
\overline{\sigma_i}:R\otimes S\to R\otimes S, \ \ \ \ \
\overline{\sigma_i}(r\otimes s)=\begin{cases}
\sigma_i(r)\otimes s, &\ 1\le i\le n\\
r\otimes \sigma'_i(r), &\ n+1 \le i \le n+m, 
\end{cases}
\]
and,
\[
\overline{\delta_i}:R\otimes S\to R\otimes S,  \ \ \ \ \
\overline{\delta_i}(r\otimes s)=\begin{cases}
\delta_i(r)\otimes s, &\ 1\le i\le n\\
r\otimes \delta'_i(s), &\ n+1 \le i \le n+m,
\end{cases}
\]
respectively. Note that the functions $\overline{\sigma_i}$ are injective endomorphisms because $\sigma_i$ and $\sigma'_j$ so are. Nex, we show that  $\overline{\delta_i}$ is a $\overline{\sigma_i}$-derivation for $1\le i \le n$:
\begin{align*}
\overline{\delta_i}((r\otimes s)(r'\otimes s')) = &\ \overline{\delta_i}(rr'\otimes ss') \\
= &\ \delta_i(rr')\otimes ss'\\
= &\ (\sigma_i(r)\delta_i(r') + \delta_i(r)r')\otimes ss'\\
= &\ \sigma_i(r)\delta_i(r')\otimes ss' + \delta_i(r)r'\otimes ss'\\
= &\ (\sigma_i(r)\otimes s)(\delta_i(r')\otimes s) + (\delta_i(r)\otimes s)(r'\otimes s')\\
= &\ \overline{\sigma_i}(r\otimes s)\overline{\delta_i}(r'\otimes
s') + \overline{\delta_i}(r\otimes s)(r'\otimes s').
\end{align*}
Similarly, we can see that $\overline{\delta_i}$ is a $\overline{\sigma_i}$-derivation, for $n+1\le i \le n+m$.
\end{proof}
\end{proposition}
\begin{remark}
Note that if $f$ is a nonzero element of $A$ and $g$ is a nonzero element of $A'$, then  ${\rm exp}(f\otimes g)=({\rm exp}(f), {\rm exp}(g))\in \mathbb{N}^{n+m}$, where ${\rm exp}(f\otimes g)$ is obtained using an order of elimination, either $A$ or $A'$.
\end{remark}
\section{Enveloping algebra}\label{envelopingalgebra} 
The {\em opposite} of a ring is the ring with the same elements and addition operation, but with
the multiplication performed in the reverse order. More precisely, the opposite of a ring $(B, +, \cdot)$ is the ring $(B, +, *)$,
whose multiplication $*$ is defined by $a* b = b \cdot a$. In this section we show that the enveloping algebra of a bijective skew PBW extension $A$ is again a PBW extension. We recall that if $B$ is a $k$-algebra, then  the \emph{enveloping algebra} of $B$ is  $B^e:=B\otimes_k B^{\rm op}$, where $B^{\rm op}$ is the opposite algebra of $B$. 
\begin{proposition}\label{opbijectiveprop}
If  $A$ is a bijective skew PBW extension over $R$, then $A^{\rm op}$ is a bijective skew PBW extension over $R^{\rm op}$. In fact, for  $A^{\rm op}$ we have the automorphisms $\sigma_i^{\rm op}:R^{\rm op}\to R^{\rm op}$ given by $\sigma^{\rm op}_i:=\sigma_i^{-1}(r)$, and the $\sigma_i^{\rm op}$-derivations $\delta^{\rm op}_i:R^{\rm op}\to R^{\rm op}$ defined by $\delta^{\rm op}_i(r):=-\delta_i(\sigma_i^{-1}(r))$, for every element $r\in R^{\rm op}$.
\begin{proof}
Let $A=\sigma(R)\langle x_1\dots, x_n\rangle$ be a bijective skew PBW extension of $R$. We will verify  the four conditions of the Definition \ref{def.pbwextension} for the rings $R^{\rm op}$ and $A^{\rm op}$.
 \begin{enumerate}
 \item[\rm(i)] It is clear that $R^{\rm op}\subseteq A^{\rm op}$.
 \item[\rm(ii)] Since $A$ is a left  free $R$-module with basis ${\rm Mon}(A):= \{x_1^{\alpha_1}\cdots
x_n^{\alpha_n}\mid  \alpha_1,\dots ,\alpha_n)\in \mathbb{N}^n\}$, then by the definition of the product in  $A^{\rm op}$, we have that $A^{\rm op}$ is a free right  $R$-module with basis the set ${\rm Mon}(A^{\rm op}):= \{x^{\alpha^{\rm op}}=x_n^{\alpha_n}\cdots
x_1^{\alpha_1}\mid \alpha^{\rm op}=(\alpha_n,\dots ,\alpha_1)\in
\mathbb{N}^n\}$. Hence, $A^{\rm op}$ is a left free $R^{\rm op}$-module.
\item[\rm(iii)] We will see that for each  $1\le i\le n$, and for every $r\in R^{\rm op}\ \backslash\ \{0\}$, there exists  $c'_{i,r}\in R^{\rm op}\ \backslash\ \{0\}$ such that $rx_i-x_ic'_{i,r}\in R^{\rm op}$. Put $c'_{i,r}:= \sigma^{-1}_i(r)$. Given that
\[
x_ic'_{i,r} = x_i\sigma_i^{-1}(r) = \sigma_i(\sigma_i^{-1}(r))x_i
+ \delta_i(\sigma_i^{-1}(r)) = rx_i + \delta_i(\sigma_i^{-1}(r)),
\]
we have that $rx_i-x_ic'_{i,r} = -\delta_i(\sigma_i^{-1}(r))\in R^{\rm op}$.

\item[\rm(iv)] Let $c'_{i,j}:=\sigma_i^{-1}(\sigma_j^{-1}(c_{i,j}^{-1}))$. Then 
\begin{align}
x_ix_j - x_jx_ic_{i,j}' = &\ x_ix_j - x_jx_i\sigma_i^{-1}(\sigma_j^{-1}(c_{i,j}^{-1}))\notag \\
= &\ x_ix_j - x_j[\sigma_i(\sigma_i^{-1}(\sigma_j^{-1}(c_{i,j}^{-1})))x_i + \delta_i(\sigma_i^{-1}(\sigma_j^{-1}(c_{i,j}^{-1})))]\notag \\
= &\ x_ix_j - x_j\sigma_j^{-1}(c_{i,j}^{-1})x_i - x_j\delta_i(\sigma_i^{-1}(\sigma_j^{-1}(c_{i,j}^{-1})))\notag \\
= &\ x_ix_j - [\sigma_j(\sigma_j^{-1}(c_{i,j}^{-1}))x_j + \delta_j(\sigma_j^{-1}(c_{i,j}^{-1}))]x_i - x_j\delta_i(\sigma_i^{-1}(\sigma_j^{-1}(c_{i,j}^{-1})))\notag \\
= &\ x_ix_j - c_{i,j}^{-1}x_jx_i - \delta_j(\sigma_j^{-1}(c_{i,j}^{-1}))x_i - x_j\delta_i(\sigma_i^{-1}(\sigma_j^{-1}(c_{i,j}^{-1})))\label{filof}
\end{align}
From Definition \ref{def.pbwextension} (iv),  we have that $x_jx_i - c_{i,j}x_ix_j = r^{(i,j)} + \sum_{l=1}^{n} r^{(i,j)}_{l}x_l$, whence $c_{i,j}^{-1}x_jx_i = x_ix_j + c_{i,j}^{-1}r^{(i,j)} + \sum_{l=1}^{n} c_{i,j}^{-1}r^{(i,j)}_{l}x_l$. So,  by replacing the term $c_{i,j}^{-1}x_jx_i$ in the above expression (\ref{filof}), we have that
\begin{align*}
x_ix_j - x_jx_ic_{i,j}' = &\ - c_{i,j}^{-1}r^{(i,j)} - \biggl( \sum_{l=1}^{n} c_{i,j}^{-1}r^{(i,j)}_{l}x_l\biggr) - \delta_j(\sigma_j^{-1}(c_{i,j}^{-1}))x_i - x_j\delta_i(\sigma_i^{-1}(\sigma_j^{-1}(c_{i,j}^{-1})))\\
= &\ -c_{i,j}^{-1}r^{(i,j)} - \biggl(\sum_{l=1}^{n} x_l\sigma_l^{-1}(c_{i,j}^{-1}r^{(i,j)}_{l}) - \delta_l(\sigma_l^{-1}(c_{i,j}^{-1}r^{(i,j)}_{l}))\biggr) \\
- &\ [x_i\sigma_i^{-1}(\delta_j(\sigma_j^{-1}(c_{i,j}^{-1}))) - \delta_i(\sigma_i^{-1}(\delta_j(\sigma_j^{-1}(c_{i,j}^{-1}))))] - x_j\delta_i(\sigma_i^{-1}(\sigma_j^{-1}(c_{i,j}^{-1})))\\
= &\ -c_{i,j}^{-1}r^{(i,j)} + \biggl(\sum_{l=1}^{n} \delta_l(\sigma^{-1}_l(c_{i,j}^{-1}r^{(i,j)}_{l}))\biggr) + \delta_i(\sigma_i^{-1}(\delta_j(\sigma_j^{-1}(c_{i,j}^{-1}))))\\
- &\ \sum_{l=1,\ l\neq i, j} x_l\sigma_l^{-1}(c_{i,j}^{-1}r_l^{(i,j)}) - x_i[\sigma_i^{-1}(c_{i,j}^{-1}r_i^{(i,j)}) + \sigma_i^{-1}(\delta_j(\sigma_j^{-1}(c_{i,j}^{-1})))]\\
- &\ x_j[\sigma_j^{-1}(c_{i,j}^{-1}r_j^{(i,j)}) + \delta_i(\sigma_i^{-1}(\sigma_j^{-1}(c_{i,j}^{-1})))],
\end{align*}
which shows that $x_ix_j-x_jx_ic_{i,j}'\in R + x_1R + \dotsb + x_nR$.
\end{enumerate}
Finally, let $r$ and $r'$ be elements of  $R^{\rm op}$. We have:
\begin{align*}
\sigma_i^{\rm op}(r+r') = &\ \sigma_{i}^{-1}(r+r') = \sigma_i^{-1}(r) + \sigma_i^{-1}(r') = \sigma_i^{\rm op}(r) + \sigma_i^{\rm op}(r')\\
\sigma_i^{\rm op}(1_{R^{\rm op}}) = &\ \sigma_i^{\rm op}(1_R) = \sigma_i^{-1}(1_R) = 1_R = 1_{R^{\rm op}} \\
\sigma_i^{\rm op}(rr') = &\ \sigma_i^{-1}(r'r) = \sigma_i^{-1}(r') \sigma_i^{-1}(r) = \sigma_i^{-1}(r)\sigma_i^{-1}(r') = \sigma_i^{\rm op}(r) \sigma_i^{\rm op}(r').
\end{align*}
Given that $\sigma_i$ is injective and surjective, so it is  $\sigma_i^{\rm op}$, for every $1\le i\le n$. With respect to the functions $\delta_i^{\rm op}$, we have
\begin{align*}
\delta_i^{\rm op}(r+r') =  &\ -\delta_i(\sigma_i^{-1}(r+r')) = -\delta_i(\sigma_i^{-1}(r) + \sigma_i^{-1}(r')) \\
= &\ -\delta_i(\sigma_i^{-1}(r)) - \delta_i(\sigma_i^{-1}(r')) \\
= &\ \delta_i^{\rm op}(r) + \delta_i^{\rm op}(r'),
\end{align*}
and using the product on $R^{\rm op}$, 
\begin{align*}
\delta_i^{\rm op}(rr') = &\ - \delta_i(\sigma_i^{-1}(r'r)) = -\delta_i(\sigma_i^{-1}(r') \sigma_i^{-1}(r))\\
= &\ - [\sigma_i(\sigma_i^{-1}(r'))\delta_i(\sigma_i^{-1}(r')) + \delta_i(\sigma_i^{-1}(r'))\sigma_i^{-1}(r)]\\
= &\ -r'\delta_i(\sigma_i^{-1}(r')) - \delta_i(\sigma_i^{-1}(r'))\sigma_i^{-1}(r)\\
= &\ \sigma_i^{-1}(r) (-\delta_i(\sigma_i^{-1}(r'))) + (-\delta_i(\sigma_i^{-1}(r)))r'\\
= &\ \sigma_i^{\rm op}(r)\delta_i^{\rm op}(r') + \delta_i^{\rm op}(r)r',
\end{align*}
which concludes the proof.
\end{proof}
\end{proposition}
\begin{remark}
We note also that $A^{\rm op}$ is a skew PBW extension over  $R$, where the elements of the  Definition \ref{def.pbwextension} are written in reverse order, and $\preceq^{\rm op}$ is the order given by $\alpha \preceq^{\rm op} \beta$ if and only if $\alpha^{\rm op}\le \beta^{\rm op}$. So, we can be that the set ${\rm Mon}(A^{\rm op})=\{x_n^{\alpha_n}\dotsb x_1^{\alpha_1}\mid \alpha^{\rm op}=(\alpha_n, \dotsc, \alpha_1)\in \mathbb{N}^{n}\}$ is a free $R$-basis of  $A^{\rm op}$.
\end{remark}

\begin{theorem}\label{envolventedeextension}
If $A$ is  a bijective skew PBW extension over $R$, then  $A^{\rm e}$ is a bijective skew PBW extension over $R^{\rm e}$.
\begin{proof}
First of all, it is clear that $R^{\rm e}\subseteq A^{\rm e}$. Secondly, given that  $A$ is a left free $R$-module, we have  that \[
A\cong R^{|{\rm Mon}(A)|}\cong (R\otimes_R R)^{|{\rm Mon}(A)|} \cong (R\otimes_{R} R^{\rm op})^{|{\rm Mon}(A)|}.
\]
Similarly, since that $A^{\rm op}$ is right $R^{\rm op}$-free,
\[
A^{\rm op}\cong (R^{\rm op})^{|{\rm Mon}(A^{\rm op})|} \cong (R^{\rm op} \otimes_{R^{\rm op}} R^{\rm op})^{|{\rm Mon}(A^{\rm op})|}\cong (R\otimes_R R^{\rm op})^{|\rm Mon(A^{\rm op})|},
\]
which shows that $A$ and $A^{\rm op}$ are left  free $R\otimes_R R^{\rm op}$-modules, so $A^{\rm e}=A\otimes A^{\rm op}$ is also a left free $R\otimes R^{\rm op}$-module with basis ${\rm Mon}(A)\otimes {\rm Mon}(A^{\rm op})$. Hence, 
\[
{\rm Mon}(A^{\rm e}) = \{(x_1\otimes 1)^{\alpha_1}\dotsb (x_n\otimes 1)^{\alpha_n} (1\otimes x_n)^{\alpha_{n+1}} \dotsb (1\otimes x_1)^{\alpha_{2n}}\mid (\alpha_1,\dotsc, \alpha_{2n})\in \mathbb{N}^{2n}\}.
\]
Note that the automorphisms $\overline{\sigma_i}$ and the  $\overline{\sigma_i}$-derivations of $A^{\rm e}$, for $1\le i\le n$, are given by
\[
\overline{\sigma_i}:R\otimes R^{\rm op}\to R\otimes R^{\rm op}, \ \ \ \ \
\overline{\sigma_i}(r\otimes r')=\begin{cases}
\sigma_i(r)\otimes r', &\ 1\le i\le n\\
r\otimes \sigma'_i(r), &\ n+1 \le i \le 2n,
\end{cases}
\]
and,
\[
\overline{\delta_i}:R\otimes R^{\rm op}\to R\otimes R^{\rm op},  \ \ \ \ \
\overline{\delta_i}(r\otimes r')=\begin{cases}
\delta_i(r)\otimes r', &\ 1\le i\le n\\
r\otimes \delta'_i(r'), &\ n+1 \le i \le 2n.
\end{cases}
\]
In this way, the conditions (iii) and (iv) of Definition \ref{def.pbwextension} follow from Propositions \ref{tensorproductsamering} and \ref{opbijectiveprop}, and Remark \ref{dogma}.
\end{proof}
\end{theorem}
\section{Skew Calabi-Yau algebras}\label{application}
Suppose that $M$ and $N$ are both $B^e$-modules. Then there are
two $B^e$-module structures on $M\otimes N$. One of them is called the {\em outer structure} defined by
$(a\otimes b)\cdot  (m\otimes n)  \overset{\rm out}= am\otimes nb,$ and the other is called the {\em inner structure} defined by
$(a\otimes b)\cdot(m\otimes n) \overset{\rm int}= ma\otimes bn,$
for any $a, b \in B$, $m\in M$, $n\in  N$. Since $B^e$ is identified with $B\otimes B$ as a 
$\Bbbk$-module ($_\Bbbk B^e =\ _{\Bbbk}(B \otimes B^{op}) = {_{\Bbbk}(B \otimes B)}$), $B\otimes B$ endowed with the outer structure is nothing but the left regular $B^e$-module $B^e$. ${_{B^e} (B\otimes B)} \overset{\rm out}= {_{B^e}B^e}$: for ${_{B^e}(B\otimes B)}$,  $(a\otimes b)\cdot (x\otimes y)= a\cdot(x\otimes y)\cdot b \overset{\rm out}= ax \otimes yb$, whereas that in ${_{B^e}B^e}$ $(a\otimes b)\cdot (x\otimes y) = ax \otimes b\circ y= ax \otimes yb$. $B\otimes B$ endowed with the  inner structure is nothing but the  right regular $B^e$-module $B^e$. ${_{B^e}(B\otimes B)} \overset{\rm int}={B^e_{B^e}}$: for ${_{B^e}(B\otimes B)}$, $(a\otimes b)\cdot (x\otimes y)= a\cdot(x\otimes y)\cdot b \overset{\rm int}=xa \otimes by$, whereas that in ${B^e_{B^e}}$,  $(x\otimes y)\cdot (a\otimes b)= xa \otimes y\circ b= xa \otimes by$. Hence, we often say that $B^e$ has the outer (left) and inner (right) $B^e$-module structure.\\

\noindent An algebra $B$ is said to be \emph{homologically smooth} if as an $B^e$-module, $B$ has a finitely ge\-ne\-ra\-ted
projective resolution of finite length. The length of this resolution is known as the
\emph{Hochschild dimension} of $B$ (in \cite{ReyesSuarez2016c}, the authors considered this dimension to compute the cyclic homology of skew PBW extensions). In the next definition, the
outer structure on $B^e$ is used when computing the homology ${\rm Ext}^*_{B^e}(B, B^e)$. Thus,
${\rm Ext}^*_{B^e}(B, B^e)$ admits an $B^e$-module structure induced by the inner one on $B^e$.
\begin{definition}\label{def1.1} An algebra $B$ is called \emph{skew Calabi-Yau} of dimension $d$, if the following conditions hold:
\begin{enumerate}
\item[(i)] $B$ is homologically smooth.
\item[(ii)]There exists an algebra automorphism  $\nu$ of $B$ such that 
${\rm Ext}^i_{B^e} (B,B^e) \cong \begin{cases} 0, &\ i\neq d\\ B^{\nu}, &\ i=d, \end{cases}$, as $B^e$-modules.  If $\nu$ is the identity, then $B$ is said to be \emph{Calabi-Yau}.
\end{enumerate}
The automorphism $\nu$ is called the \emph{Nakayama} automorphism of $B$, and it is unique up to inner automorphisms of $B$. Note that a skew Calabi-Yau algebra is Calabi-Yau if and only if its Nakayama automorphism is inner. 
\end{definition}
\begin{definition}\label{ASregulardefinition}
Let $B=\Bbbk\oplus B_1\oplus B_2\oplus \cdots$ be a finitely presented graded algebra over a field $\Bbbk$. The algebra $B$ will be called \emph{AS-regular},  if it has the
following properties:
\begin{enumerate}
\item [(i)] $B$ has finite global dimension $d$, i.e., every graded $B$-module has projective
dimension less or equal than $d$; 
\item [(ii)] $B$ has finite Gelfand-Kirillov dimension; 
\item [(iii)] $B$ is \emph{Gorenstein}, meaning that ${\rm Ext}_B^i(\Bbbk, B) = 0$ if $i \neq d$, and ${\rm Ext}^d_A(\Bbbk, B)\cong \Bbbk$.
\end{enumerate}
\end{definition}
The dimension of Gelfand-Kirillov and the notion of Gorenstein for skew PBW extensions were studied in \cite{Reyes2013} and \cite{ReyesPhD}, respectively. Now, from the Definition \ref{ASregulardefinition} we can see that we need to consider graded algebras, and since in general skew PBW extensions are not graded rings, in the next definition we impose three conditions to guarantee a notion of {\em grade} in these extensions. More exactly,
\begin{definition}[\cite{Suarez}, Definition 2.6]\label{def. graded skew PBW ext} Let  $A=\sigma(R)\langle x_1,\dots, x_n\rangle$ be a bijective skew PBW extension over a  $\mathbb{N}$-graded algebra $R$. We said that $A$ is a \emph{graded  skew PBW extension}, if the following conditions holds:
\begin{enumerate}
\item[\rm (i)] the indeterminates $x_1,\dots, x_n$ have degree 1 in $A$; 
\item[\rm (ii)] $\sigma_i$ is a graded ring homomorphism and $\delta_i : R(-1) \to R$ is a graded $\sigma_i$-derivation, for all $1\leq i  \leq n$, where $\sigma_i$ and $\delta_i$ are established in Proposition \ref{sigmadefinition}; 
\item[\rm (iii)]  $x_jx_i-c_{i,j}x_ix_j\in R_2+R_1x_1 +\cdots + R_1x_n$, as in Definition \ref{def.pbwextension} (iv), and $c_{i,j}\in R_0$; 
\end{enumerate}
\end{definition}
For the next proposition, consider the notation established in Definition \ref{definitioncoefficients}. 
\begin{proposition}[\cite{Suarez}, Proposition 2.7]\label{prop.grad A}
Let $A$ be a graded skew PBW extension over $R$, and let $A_p$ the $\Bbbk$-space generated by the set $\{r_tx^{\alpha} \mid t+|\alpha|= p,\  r_t\in R_t \text{  and } x^{\alpha}\in {\rm Mon}(A)\}$, for $p\geq 0$. Then $A$ is a  graded algebra with graduation given by $A=\bigoplus_{p\geq 0} A_p$.
\end{proposition}
Next theorem is one of the most important results of this paper. This theorem establishes that quasi-commutative skew PBW extensions over connected Calabi-Yau algebras over fields are skew-Calabi-Yau.
\begin{theorem}\label{teo.skew imp skew} 
If $A$ is a graded quasi-commutative skew PBW extension over a connected skew Calabi-Yau $\Bbbk$-algebra $R$, then $A$ is skew Calabi-Yau.
\end{theorem}
\begin{proof}
Note that $A$ is  isomorphic to an  iterated Ore extension of endomorphism type \linebreak $R[z_1, \theta_1]\cdots[z_n, \theta_n]$ where $\theta_i$ is bijective; $\theta_1=\sigma_1$; $$\theta_j: R[z_1;\theta_1]\cdots[z_{j-1};\theta_{j-1}]\to R[z_1;\theta_1]\cdots[z_{j-1};\theta_{j-1}]$$ is such that $\theta_j(z_i)=c_{i,j}z_i$ ($c_{i,j}\in R$ as in Definition \ref{def.pbwextension}), $1 \leq i < j \leq n$ and $\theta_i(r)=\sigma_i(r)$, for $r\in R$ (\cite{LezamaReyes2014}, Theorem 2.3).  Since  $A$ is graded, then $\sigma_i$ is graded and $c_{i,j}\in R_0$. Now, using that $\theta_i(r)=\sigma_i(r)$ and $\theta_j(z_i)=c_{i,j}z_i$, we have that $\theta_i$ is a graded automorphism, for every  $i$.  Without loss of generality, we can assume that $z_i=x_i$, for every $1\le i\le n$. Therefore $A$ is isomorphic to a graded iterated Ore extension, and using that $R$ is AS-regular (\cite{Reyes}, Lemma 1.2), then $A$ is AS-regular (\cite{Liu1}). Now, $R$ connected implies that $A$ so is  (\cite{Suarez}, Remark 2.10). From \cite{Reyes}, Lemma 1.2, we conclude that $A$ is skew Calabi-Yau.
\end{proof}
\begin{example}\label{ex.skew CY} Theorem \ref{teo.skew imp skew} allows us to obtain the following examples of skew PBW extensions which are skew Calabi-Yau algebras.
\begin{enumerate}
\item \label{ex.q-dilat oper} For a fixed $q\ \in\ \Bbbk \setminus \{0\}$, the
$\Bbbk$-\emph{algebra of linear partial $q$-dilation operators} with polynomial coefficients  is $\Bbbk[t_1,\dots ,t_n][H^{(q)}_1, \dots, H^{(q)}_m ]$, $n \geq m$, subject to the relations:
$t_jt_i = t_it_j$, $1 \leq i < j \leq n$;\quad $H^{(q)}_i t_i = qt_iH^{(q)}_i$, $1 \leq i \leq m$;\quad
$H^{(q)}_j t_i = t_iH^{(q)}_j$, $i \neq j$;\quad
$H^{(q)}_j H^{(q)}_i = H^{(q)}_iH^{(q)}_j$, $1 \leq i < j\leq m$ (see \cite{LezamaReyes2014}, section 3.3). This algebra is a graded quasi-commutative skew PBW extension of $\Bbbk[t_1,\dots ,t_n]$, where $\Bbbk[t_1,\dots ,t_n]$ is endowed with usual graduation.
\item \label{ex.multipl Weyl} The \emph{quantum polynomial ring} $\mathcal{O}_n(\lambda_{ji})$  (also known as the {\em multiplicative analogue of the Weyl algebra}) is the algebra  generated by the indeterminates 
$x_1,\dots,x_n$ subject to the relations $x_jx_i
=\lambda_{ji}x_ix_j ,\ 1\leq i<j\leq n$, $\lambda_{ji}\in
\Bbbk\setminus \{0\}$. In \cite{LezamaReyes2014}, Section 3.5, it was proved that 
$\mathcal{O}_n(\lambda_{ji})\cong\sigma(\Bbbk)\langle
x_1,\dotsc,x_n\rangle\cong \sigma(\Bbbk[x_1])\langle
x_2,\dotsc,x_n\rangle$.
\end{enumerate}
\end{example}

\subsection*{Acknowledgment}
The first author is supported by Grant HERMES CODE 30366, Departamento de Matem\'aticas, Universidad Nacional de Colombia, Bogot\'a.

\end{document}